\documentclass{amsart}
\usepackage{amssymb,amsmath,amsthm,graphicx,hyperref,subfigure}
\theoremstyle{plain}
\newtheorem{thm}{Theorem}[section]
\newtheorem{lm}[thm]{Lemma}
\newtheorem{prop}[thm]{Proposition}
\newtheorem{cor}[thm]{Corollary}
\newtheorem{que}[thm]{Question}
\theoremstyle{definition}
\newtheorem{ex}[thm]{Example}
\newtheorem{re}[thm]{Remark}

\newcommand{\CC}{{\mathbb C}}
\newcommand{\NN}{{\mathbb N}}
\newcommand{\QQ}{{\mathbb Q}}
\newcommand{\RR}{{\mathbb R}}
\newcommand{\Rb}{\RR_\infty}
\newcommand{\ZZ}{{\mathbb Z}}
\newcommand{\val}{v}
\newcommand{\lieg}[1]{\mathrm{#1}}
\newcommand{\oD}{\overline{D}}
\newcommand{\Sym}{\operatorname{Sym}}
\newcommand{\tp}{\top}
\newcommand{\tim}{$\times$}
\newcommand{\maysplit}{\discretionary{}{}{}}
\newcommand{\Trop}{\mathrm{Trop}}
\newcommand{\tdet}{\mathrm{tdet}}

\begin{document}

\title{Lossy gossip and composition of metrics}

\begin{abstract}
We study the monoid generated by $n \times n$ distance matrices under
tropical (or min-plus) multiplication. Using the tropical geometry of the
orthogonal group, we prove that this monoid is a finite polyhedral
fan of dimension $\binom{n}{2}$, and we compute the structure of
this fan for $n$ up to $5$. The monoid captures gossip among $n$
gossipers over lossy phone lines, and contains the gossip monoid over
ordinary phone lines as a submonoid. We prove several new results about
this submonoid, as well. In particular, we establish a sharp bound 
on chains of calls in each of which someone learns
something new.
\end{abstract}

\author{Andries E.~Brouwer, Jan Draisma, and Bart J.~Frenk}
\address{
Department of Mathematics and Computer Science\\
Technische Universiteit Eindhoven\\
P.O. Box 513, 5600 MB Eindhoven, The Netherlands}
\thanks{JD is supported by a Vidi grant from the Netherlands Organisation
for Scientific Research (NWO), and BJF by an NWO free competition
grant.}
\email{aeb@cwi.nl, j.draisma@tue.nl, b.j.frenk@tue.nl}
\maketitle

\section{Introduction and results}

\begin{figure}
\includegraphics[width=\textwidth]{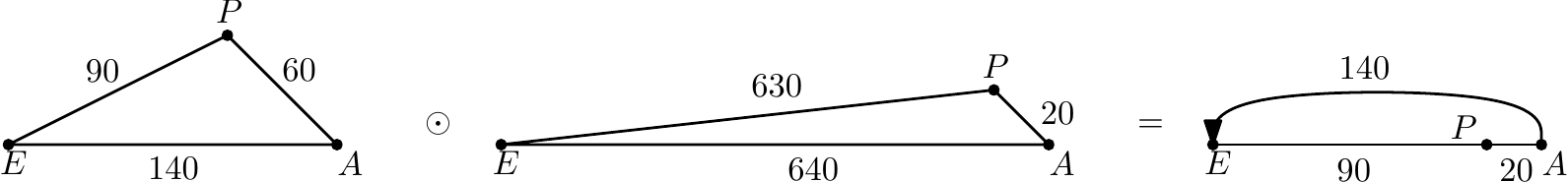}
\caption{Composing the car metric with the bike metric.}
\label{fig:EPA}
\end{figure} 

Imagine travelling between three locations such as Eindhoven ($E$, a
medium-sized town in the Netherlands), a parking lot $P$ on the border
of the Dutch capital Amsterdam, and the city center $A$ of Amsterdam.
In Figure~\ref{fig:EPA} the travel times by car between these locations
are depicted by the leftmost triangle, while the travel times by bike
are depicted by the second triangle. The large distances
between $E$ and either $P$ or $A$ are covered much faster by car than
by bike. On the other hand, because of crowded streets, the short
distance between $P$ and $A$ is covered considerably faster by bike than
by car.  As a consequence, an attractive alternative for travelling
from $E$ to $A$ by car is to travel by car from $E$ to $P$ and
continue by bike to $A$.
In other words, to get from $E$ to $A$ we first do a step in
the car metric and then a step in the bike metric, where we optimise
the sum of the two travel times. Computing this {\em car-bike metric}
for the remaining ordered pairs leads to the picture on the
right in Figure~\ref{fig:EPA}. The corresponding matrix
computation is 
\[ 
\begin{bmatrix}
0 & 90 & 140\\
90 & 0 & 60 \\
140 & 60 & 0
\end{bmatrix}
\odot 
\begin{bmatrix}
0 & 630 & 640\\
630 & 0 & 20\\
640 & 20 & 0
\end{bmatrix}
=
\begin{bmatrix}
0 & 90 & 110\\
90 & 0 & 20\\
140 & 20 & 0
\end{bmatrix},
\] 
where $\odot$ is {\em tropical} or {\em min-plus} matrix multiplication,
obtained from usual matrix multiplication by changing plus into minimum
and times into plus. Note that the resulting matrix is not symmetric (the
transpose corresponds to the ``first bike, then car'' metric), and that
it does not satisfy the triangle inequality either. The bike metric and
the car metric were both picked from the $3$-dimensional cone of symmetric
matrices satisfying all triangle inequalities. Hence one might think that
such tropical products sweep out a $3+3=6$-dimensional set. However, if
we perturb the travel times in the two metric matrices slightly, then
their min-plus product moves only in a {\em three}-dimensional space,
where the entry at position $(1,3)$ remains the sum of the entries
in positions $(1,2)$ and $(2,3)$, while the entry in position $(3,1)$
moves freely.  This preservation of dimension when tropically multiplying
cones of distance matrices is one of the key results of this paper.

While keeping this min-plus product in the back of our minds, we next
contemplate the following different setting. Three gossipers, Eve,
Patricia, and Adam, each have an individual piece of gossip, which
they can share through one-to-one phone calls in which both callers
update each other on all the gossip they know.  Record the knowledge of
$E,P,A$ in a three-by-three uncertainty matrix with entries $0$
(for ``$i$'s gossip is known by $j$'') and $\infty$ (for the other entries).
Then initially that matrix is the {\em tropical identity matrix},
with zeroes along the diagonal and $\infty$ outside the diagonal.
A phone call beteen $E$ and $P$, for example,
corresponds to tropically right-multiplying that tropical
identity matrix matrix with
\[ \begin{bmatrix} 0 & 0 & \infty \\ 0 & 0 & \infty \\
\infty & \infty & 0 \end{bmatrix}, \]
resulting in this very same matrix. A second phone call
between $P$ and $A$ leads to 
\[ \begin{bmatrix} 0 & 0 & \infty \\ 0 & 0 & \infty \\
\infty & \infty & 0 \end{bmatrix} \odot 
\begin{bmatrix} 0 & \infty & \infty \\ \infty & 0 & 0 \\
\infty & 0 & 0 \end{bmatrix} = 
\begin{bmatrix} 0 & 0 & 0  \\ 0 & 0 & 0 \\ \infty & 0 & 0
\end{bmatrix}. \]
Note the resemblance of this computation with the car-bike metric
computation above.  This resemblance can be made more explicit by passing
from gossip to {\em lossy gossip}, where each phone call between gossipers
$k$ and $l$ comes with a parameter $q \in [0,1]$ to be interpreted as
the fraction of information that gets broadcast correctly through the
phone line, and where each gossiper $j$ knows a fraction $p_{ij} \in
[0,1]$ of $i$'s gossip. Assume the (admittedly simplistic) procedure
where $k$ updates his knowledge of gossip $i$ to $q \cdot p_{il}$ if this
is larger than $p_{ik}$ and retains his knowledge $p_{ik}$ of gossip $i$
otherwise, and similarly for gossiper $l$. In this manner, the fractions
$p_{ij}$ are updated through a series of lossy phone-calls.  Passing from
$p_{ij}$ to the {\em uncertainty} $u_{ij}:=-\log p_{ij} \in [0,\infty]$
of gossiper $j$ about gossip $i$ and from $q$ to the {\em loss} $a:=-\log
q \in [0,\infty]$ of the phone line in the call between $k$ and $l$, the
update rule changes $u_{ik}$ into the minimum of $u_{ik}$ and $u_{il}+a$,
and similarly for $u_{il}$. This is just tropical right-multiplication
with the matrix $C_{kl}(a)$ having $0$'s on the diagonal, $\infty$'s
everywhere else, except an $a$ on positions $(k,l)$ and $(l,k)$. So {\em
lossy gossip is tropical matrix multiplication}. Note that lossy gossip
is different from {\em gossip over faulty telephone lines} discussed
in \cite{Berman86,Haddad87}, and also from gossip algorithms via
multiplication of doubly stochastic matrices as in \cite{Boyd06} (though
the elementary matrices $W_{kl}$ there are reminiscent of our matrices
$C_{kl}$).

This paper concerns the entirety of such uncertainty matrices, or
compositions of finite metrics. Our main result uses the following
notation: fixing a number $n$ (of gossipers or vertices), let $D=D_n$
be the set of all {\em metric} $n \times n$ matrices, i.e., matrices
with entries $a_{ij} \in \RR_{\geq 0}$ satisfying $a_{ii}=0$ and
$a_{ij}=a_{ji}$ and $a_{ij}+a_{jk} \geq a_{ik}$. For standard notions
in polyhedral geometry, we refer to \cite{Ziegler95}.

Throughout the paper, we give $[0,\infty]$ the topology of the one-point
compactification of $[0,\infty)$, i.e., the topology of a compact,
closed interval.

\begin{thm} \label{thm:main}
The set $\{A_1 \odot \cdots \odot A_k \mid k \in \NN,\ A_1,\ldots,A_k
\in D_n\}$ is the support of a (finite) polyhedral fan of dimension
$\binom{n}{2}$, whose topological closure in $[0,\infty]^{n \times
n}$ (with product topology) is the monoid generated by the matrices
$C_{kl}(a)$ with $k,l \in [n]:=\{1,\ldots,n\}$ and $a \in [0,\infty]$.
\end{thm}

We will denote that monoid by $G_n$, and call it the {\em lossy gossip
monoid} with $n$ gossipers.
The most surprising part of this theorem is that the dimension
of $G_n$ is not larger than $\binom{n}{2}$. We will establish this in
Section~\ref{sec:Proof} by proving that $G_n$ is contained the
tropicalisation of the orthogonal group $\lieg{O}_n$.

\begin{thm} \label{thm:n234}
For $n \leq 5$ the fan in the previous theorem is pure and
connected in codimension $1$. Moreover, for $n \leq 4$, there is a unique
coarsest such fan. This coarsest fan has $D_2,D_3,D_4$ among the
$1,7,289$ full-dimensional cones; and in total it has $1,2,16$ orbits
of full-dimensional cones under the groups $\Sym(2),\Sym(3),\Sym(4)$,
respectively.
\end{thm}

For some statistics for $n=5$ we refer to Section~\ref{sec:Five}.
We conjecture that the pureness and connectedness in codimension $1$
carry through to arbitrary $n$. 

About the length of products we can say the following.

\begin{thm}
For $n  \le 5$ every element of $G_n$ is the tropical product
of at most $\binom{n}{2}$ lossy phone call matrices $C_{kl}(a)$,
but not every element is the tropical product of fewer factors.
\end{thm}

We conjecture that the restriction $n \le 5$ can be omitted.

Our next result concerns ``pessimal'' ordinary gossip (the least efficient
way to spread information, keeping the gossipers entertained for as long
as possible).

\begin{thm} \label{thm:pessimal}
Any sequence of phone calls among $n$ gossiping parties such that in each
phone call both participants exchange all they know, and at least one
of the parties learns something new, has length at most $\binom{n}{2}$,
and this bound is attained.
\end{thm}

This implies a bound on the length of {\em irredundant products} of
matrices $C_{kl}(0)$, i.e., tropical products where leaving out any
factor changes the value of the product.

\begin{cor} \label{cor:irredundant}
In the monoid generated by the matrices $C_{kl}(0),\ k,l \in [n]$ every
irredundant product of such matrices has at most $\binom{n}{2}$ factors.
\end{cor}

\medskip
Our motivation for this paper is twofold. First, it establishes a
connection between gossip networks and composition of metrics that seems worth
pursuing further. Second, the lossy gossip monoid is
a beautiful example of a submonoid of $(\RR \cup \{\infty\})^{n \times
n}$; a general theory of such submonoids also seems very
worthwhile. Note
that sub{\em groups} of this semigroup (but with identity element an
arbitrary idempotent matrix) have been investigated in \cite{Izhakian12}.

The remainder of this paper is organised as
follows. Sections~\ref{sec:Prel} and~\ref{sec:Detours} contain
observations that pave the way for the analysis for $n=3,4$ in
Sections~\ref{sec:Three} and~\ref{sec:Four}. In Section~\ref{sec:Five} we
report on extensive computations for $n=5$. In Section~\ref{sec:Proof}
we discuss tropicalisations of the special linear groups and the
orthogonal groups, and use the latter to prove the first statement of
Theorem~\ref{thm:main}. Interestingly, no polyhedral-combinatorial proof
of Theorem~\ref{thm:main} is known. In Section~\ref{sec:Ordinarygossip}
we study the monoid generated by the ordinary gossip matrices
$C_{kl}(0),\ k,l \in [n]$: using the ordinary orthogonal group we
prove Theorem~\ref{thm:pessimal}, and for $n \leq 9$ we determine the
order of this monoid. We conclude with a number of open questions in
Section~\ref{sec:Open}.

\section*{Acknowledgments}
We thank Tyrrell McAllister for discussions on the tropical
orthogonal group many years ago; and Peter Fenner and Mark Kambites
for pointing out problems with an earlier, purely combinatorial proof
of Theorem~\ref{thm:pessimal}.

\section{Preliminaries} \label{sec:Prel}

Fixing a natural number $n$, we define $\oD_n$ to be the topological
closure of $D_n$ in $[0,\infty]^{n \times n}$, and we denote by $G_n$
the monoid generated by $\oD_n$ under min-plus matrix multiplication.
We call $G_n$ the {\em lossy gossip monoid} with $n$ gossipers. This
terminology is justified by the following lemma.

\begin{lm} \label{lm:Generation}
The lossy gossip monoid $G_n$ is generated by the {\em lossy phone call
matrices} $C_{kl}(a)\ (k,l \in [n], a \in [0,\infty])$ having
zeroes on the diagonal and $\infty$ everywhere else except for values
$a$ on positions $(k,l)$ and $(l,k)$.
\end{lm}

\begin{proof}
Lossy phone call matrices lie in $\oD_n$, so the monoid that they generate
is contained in $G_n$. For the converse it suffices to show that every
element $A$ of $\oD_n$ is the product of lossy phone call matrices.
We claim that, in fact, $A=\prod_{k<l} C_{kl}(a_{kl})=:B$, where the
$a_{kl}$ are the entries of $A$ and the product is
taken in any order. Indeed, the $(i,j)$-entry of $B$ is the minimum of
expressions of the form $a_{i_0,i_1}+a_{i_1,i_2}+\ldots+a_{i_{s-1},i_s}$
where $s \leq \binom{n}{2}$, $i_0=i$, $i_s=j$, and where the
$C_{i_0,i_1},\ldots,C_{i_{s-1},i_s}$ (with $s\leq \binom{n}{2}$) appear in
that order (though typically interspersed with other
factors) in the product expression for $B$. By the triangle inequalities
among the entries of $A$, the minimum of these expressions equals
$a_{i,j}$.
\end{proof}

Although elements of $G_n$ need not be symmetric, they have
a symmetric core.

\begin{lm}
Each element $A$ of $G_n$ satisfies $a_{ij} = a_{ji}$ for at least $n-1$
pairs of distinct indices $i,j$. The graph with vertex set $[n]$ and
these pairs as edges is connected.
\end{lm}

\begin{proof}
We need to prove that for any partition of $[n]$ into two nonempty
parts $K$ and $L$ there exist a $k \in K$ and an $l \in L$ such
that $a_{kl}=a_{lk}$. Write $A=C_{i_1,j_1}(b_1) \odot \cdots \odot
C_{i_s,j_s}(b_s)$ with $b_1,\ldots,b_s \in \RR_{\geq 0}$. If there
is no $r$ such that $i_r$ and $j_r$ lie in different sides of the
partition, then $a_{kl}=a_{lk}=\infty$ for all $k \in K$ and $l \in L$.
Otherwise, among all $r$ for which $i_r$ and $j_r$ lie in different
parts of the partition choose one for which $b_r$ is minimal. Then
$a_{i_r,j_r}=a_{j_r,i_r}=b_r$.
\end{proof}

\begin{lm}
Every connected graph on $[n]$ occurs as symmetric core of some element
of $G_n$. 
\end{lm}

\begin{proof}
Number the edges of that subgraph $I_1,\ldots,I_m$
and consider the tropical product
\begin{align*}
A=&C_{I_1}(1+2^{-1}) \odot C_{I_2}(1+2^{-2}) \odot \cdots \odot
C_{I_m}(1+2^{-m})\\
\odot &C_{I_1}(2^1) \odot C_{I_2}(2^2) \odot \cdots \odot C_{I_m}(2^m) \\
\odot &C_{I_1}(2^{m+1}) \odot \cdots \odot
C_{I_m}(2^{2m})\\
\odot &C_{I_1}(2^{2m+1}) \odot \cdots,
\end{align*}
where the product stabilises once all edges have acquired a finite length.
For $\{k,l\}$ equal to some $I_i$ we have $a_{kl}=a_{lk}=1+2^{-i}$, where
we use that the sum of two of these $m$ numbers is larger than any third.
For any other $\{k,l\}$ the value $a_{kl}$ is a sum of some number $m$
of distinct negative powers of $2$, the integer $m$ itself, and some number of
distinct positive powers of $2$. This sum uniquely determines the sequence of
factors contributing to it, of which there are at least two. Hence the sum
determines the ordered pair $(k,l)$. In particular, we have $a_{lk}
\neq a_{kl}$.
\end{proof}

\medskip
Observe that $C_{kl}(a) \odot C_{kl}(b) = C_{kl} (a \oplus b)$, where
$\oplus$ denotes tropical addition defined by $a \oplus b = \min (a,b)$.
Thus Lemma \ref{lm:Generation} exhibits $G_n$ as a
monoid generated by certain {\em one-parameter submonoids}, reminiscent
of the generation of algebraic groups by one-parameter subgroups. This
resemblance will be exploited in Sections~\ref{sec:Proof}
and~\ref{sec:Ordinarygossip}.

We define the {\em length} of an element $X$ of $G_n$ as the minimal
number of factors in any expression of $X$ as a tropical product of lossy
phone call matrices $C_{kl}(a)$. A rather crude, but uniform upper bound
on the length of elements of $G_n$ is the maximal number of factors in
a tropical product of lossy phone call matrices in which no factor can
be left out without changing the result. We call such an 
expression {\em irredundant}, and we have the following bounds. 

\begin{lm} \label{lm:UpperBound}
The number of factors in any irredundant tropical product of lossy
phone call matrices in $G_n$ is at most $n^2(n-1)/2$. In particular,
the length of every element of $G_n$ is bounded by this number.
\end{lm}

\begin{proof}
Let $A$ be an element of $G_n$ and write
\[
A = C_{I_1}(a_1) \odot \cdots \odot C_{I_k}(a_k) 
\]
where the $a_j$ are non-negative real numbers and the $I_j$ are unordered
pairs of distinct numbers in $[n]$. The entry at position $(h, i)$ of
$A$, if not equal to $\infty$, is the minimum of expressions $a_{k_1} +
\ldots + a_{k_s}$, where $(I_{k_1},\ldots,I_{k_s})$ is a path from $h$
to $i$ in the complete graph on $[n]$ and $k_1 < \cdots < k_s$. Choose
such a path with $s$ minimal, and call this the minimal path from $h$
to $i$.  Since it is never cheaper to visit a vertex twice, we have $s
\leq n-1$. This shows that for each of the $n(n-1)$ pairs $(h,i)$ only
at most $(n-1)$ of the factors are necessary, and this gives an upper
bound of $n(n-1)^2$ on the quantity in the lemma. The sharper bound in
the lemma comes from the fact that if $j \neq h,i$ lies on a minimal
path from $h$ to $i$, then $i$ does not lie on a minimal path from $h$
to $j$. Hence the total number of ordered pairs $(i,j)$ with $i,j$
unequal to $h$ and $j$ on the minimal path from $h$ to $i$ is at most
$(n-1)+\binom{n-1}{2}=\binom{n}{2}$, and this bounds the number of 
factors essential for the $h$-th row of $A$. This gives the bound. 
\end{proof}

\begin{lm} \label{lm:LowerBound}
There exists an expression that is an irredundant tropical product of
$\binom{n+1}{3}$ lossy phone call matrices in $G_n$.
\end{lm}
\begin{proof}
We proceed by induction on $n$. For $n = 1$ there are no factors.
Let $W_{n-1}$ be an irredundant expression over $G_n$ of length
$\binom{n}{3}$ not involving the index 1. Let $P_h$ be the product
\[
P_h = C_{12}(b_{h1}) \odot C_{23}(b_{h2}) \odot \cdots \odot C_{h,h+1}(b_{hh})
\]
(of length $h$) and put
\[
W_n = W_{n-1} \odot P_{n-1} \odot P_{n-2} \odot \cdots \odot P_1 .
\]
% Then $W_n$ has length $\binom{n}{3} + \binom{n}{2} = \binom{n+1}{3}$.
Then the expression for $W_n$ has length $\binom{n+1}{3}$.
Order the constants involved such that those in $W_{n-1}$ are small,
those in $P_1$ (just $b_{11}$) much larger, those in $P_2$ larger again,
and those in $P_{n-1}$ the largest. The matrix that is the result of
multiplying out the expression $W_n$ has $(i,j)$-entry as found for
$W_{n-1}$ when $i,j \ne 1$, but $(1,h+1)$-entry as found for $P_h$
(since 1 is not found in $W_{n-1}$, $h+1$ is not found later than in
$P_h$, and earlier $P_j$ are too expensive). It follows that no factor
of $P_h$ is redundant.
\end{proof}

%}

\begin{prop}
The closure of $D_n$ under tropical matrix multiplication is the support
of some finite polyhedral fan in $\RR_{\geq 0}^{n \times n}$ and
equals $G_n \cap \RR_{\geq 0}^{n \times n}$. Its topological closure in
$[0,\infty]^{n \times n}$ equals $G_n$.
\end{prop}

Note that this is Theorem~\ref{thm:main} minus the claim that
the dimension of that fan is (not more than) $\binom{n}{2}$; this
claim will be proved in Section~\ref{sec:Proof}.

\begin{proof}
By Lemma~\ref{lm:Generation} and the proof of Lemma~\ref{lm:UpperBound}
the closure of $D_n$ under tropical matrix multiplication is a finite union of
images of orthants $\RR_{\geq 0}^k$ with $k \leq n(n-1)^2$ under piecewise
linear maps. Such an image is the support of some polyhedral fan.
The remaining two statements are straightforward.
\end{proof}

From now on, we will sometimes use the term ``polyhedral fan'' for the
topological closure in $[0,\infty]^N$ of a polyhedral fan in $\RR_{\geq
0}^N$. Thus $G_n$ itself is a polyhedral fan in $[0,\infty]^{n \times n}$.

Recall that the {\em Kleene star} of $A \in [0,\infty]^{n \times n}$
is defined as
\[ 
A^*:=I \oplus A \oplus A^{\odot 2} \oplus \cdots \\
=I \oplus A \oplus A^{\odot 2} \oplus \cdots \oplus
A^{\odot(n-1)}
=(I \oplus A)^{\odot (n-1)}
\]
where $I$ is the tropical identity matrix \cite[p.  21]{Butkovic10}. The
$(i,j)$-entry of $A^*$ records the length of the shortest path from $i$
to $j$ in the directed graph on $[n]$ with edge lengths $a_{ij}$. From
this interpretation it follows readily that for $A_1,\ldots,A_s \in
[0,\infty]^{n \times n}$ with zero diagonal, and $\pi \in \Sym(s)$,
we have $(A_1 \odot \cdots \odot A_s)^* =
(A_{\pi(1)} \odot \cdots \odot A_{\pi(s)})^*$.

\begin{lm} \label{lm:Kleene}
The Kleene star maps $G_n$ into its subset $\oD_n$.
\end{lm}

\begin{proof}
Let $A \in G_n$ be the tropical product of lossy phone call
matrices $C_1,\ldots,C_k$. Note that $C_i^\tp=C_i$. We have 
\begin{align*} 
A^*&=(C_1 \odot \cdots \odot C_k)^*
=(C_k \odot \cdots \odot C_1)^*
=(C_k^\tp \odot \cdots \odot C_1^\tp)^*\\
&=((C_1 \odot \cdots \odot C_k)^\tp)^*
=((C_1 \odot \cdots \odot C_k)^*)^\tp
=(A^*)^\tp,
\end{align*}
where we have used the remark above, the fact that transposition
reverses multiplication order, and the fact that Kleene star commutes
with transposition. Thus $A^*$ is a symmetric Kleene star and hence a
metric matrix.
\end{proof}

\section{Graphs with detours} \label{sec:Detours}

In the next two sections we will visualise elements of the lossy gossip
monoids $G_3$ and $G_4$, as well as the polyhedral structures on these
monoids. We will do this through combinatorial gadgets that we dub {\em
graphs with detours}. We first recall realisations of ordinary metrics,
i.e., elements of $D_n$ (see, e.g., \cite{Dress84,Imrich84}).

Let $\Gamma = (V,E)$ be a finite, undirected graph and $w: E \rightarrow
\RR_{\geq 0}$ be a function assigning lengths to the edges of $\Gamma$. The
weight of a path in $(\Gamma, w)$ is the sum of the weights of the
individual edges in the path.  A map $\ell: [n] \rightarrow V$ is called a
\textit{labelling}, or \textit{$[n]$-labelling}, if we need to be precise,
and the pair $(\Gamma, \ell)$ is referred to as a labelled graph, or an
$[n]$-labelled graph.

A weighted $[n]$-labelled graph gives rise to a matrix $A(\Gamma, w,
\ell)$ in $D_n$ whose entry at position $(i, j)$ is the minimal weight of
a path between $\ell(i)$ and $\ell(j)$. We say that the weighted labelled
graph $(\Gamma, w, \ell)$ \textit{realises} the matrix $A(\Gamma, w,
\ell)$. Any matrix $X \in D_n$ has a realisation by some weighted,
$[n]$-labelled graph, e.g., the graph with vertex set $[n]$, the
entries of $X$ as weights, and $\ell$ equal to the identity. However,
typically more efficient realisations exist, in the following sense. A
weighted, $[n]$-labelled graph $(\Gamma=(V,E),w,\ell)$ is called an
{\em optimal realisation} of $X$ if the sum $\sum_e w(e)$ is minimal
among all realisations \cite{Imrich84}. We will, moreover, require of
an optimal realisation that no edges get weight $0$ (since such edges
can be removed and their endpoints identified), and that no vertices
in $V \setminus \ell([n])$ have valency $2$ (since such vertices can be
removed and their incident edges glued together). Optimal realisations of any $X \in D_n$
exist \cite{Imrich84}, and there is an interesting question concerning
the uniqueness of optimal realisations for generic $X$ \cite[Conjecture
3.20]{Dress84}.

Our first step in describing the cones of $G_3$ and $G_4$ is to find
weighted labelled graphs that realise the elements of $D_3,D_4$, as
follows (for much more about this see \cite{Dress84,Dress06}). We write
$J_0$ for the matrix of the appropriate size with all entries $0$.

\begin{ex}
We give optimal realisations of the elements of $D_n$, for $n
= 2,3,4$. For the cases $n = 5, 6$ see \cite{Koolen09} and
\cite{Sturmfels04}.
\begin{enumerate}
\item An element of $D_2 \setminus \{J_0\}$ is optimally realised by the
graph on two vertices having one edge with the right weight. The choice
of labelling is inconsequential as long as it is injective. The matrix
$J_0$ is optimally realised by the graph on one vertex.

\begin{figure}[ht]
\begin{center}
\includegraphics[scale=0.7]{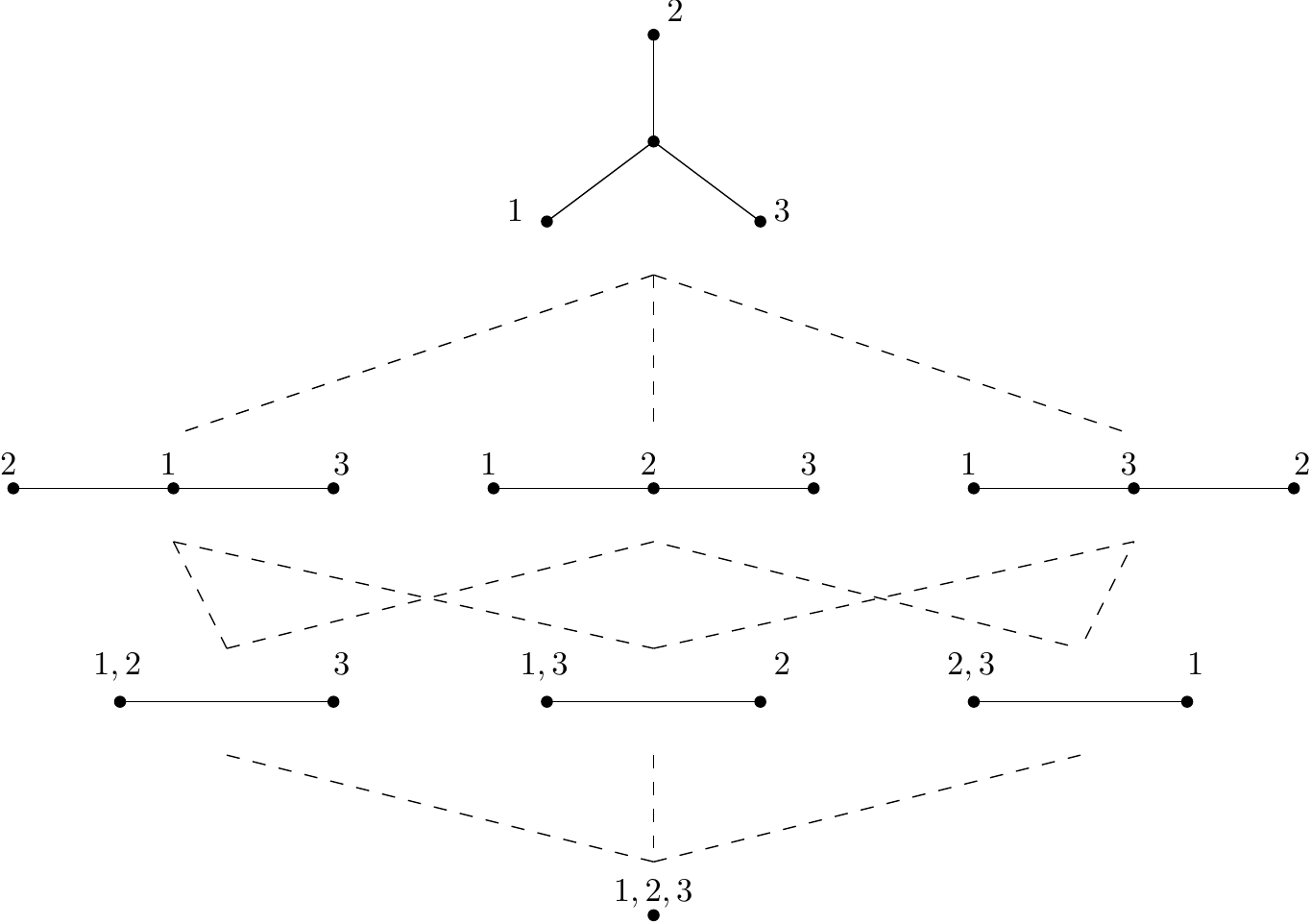}
\end{center}
\caption{\label{fig:OptPoset3}Minimal realisations 
of three-point metrics.}
\end{figure}

\item Any matrix in $D_3$ is realised by the top labelled graph of the
poset depicted in Figure \ref{fig:OptPoset3} with suitable edge weights
(note that we allow these to be zero), but only the matrices in the
relative interior of the cone $D_3$ are optimally realised by it. Matrices
on the boundary are optimally realised by some graph further down the
poset, depending on the smallest face of $D_3$ in which the matrix lies.

\begin{figure}[ht]
\begin{center}
\includegraphics[scale=0.4]{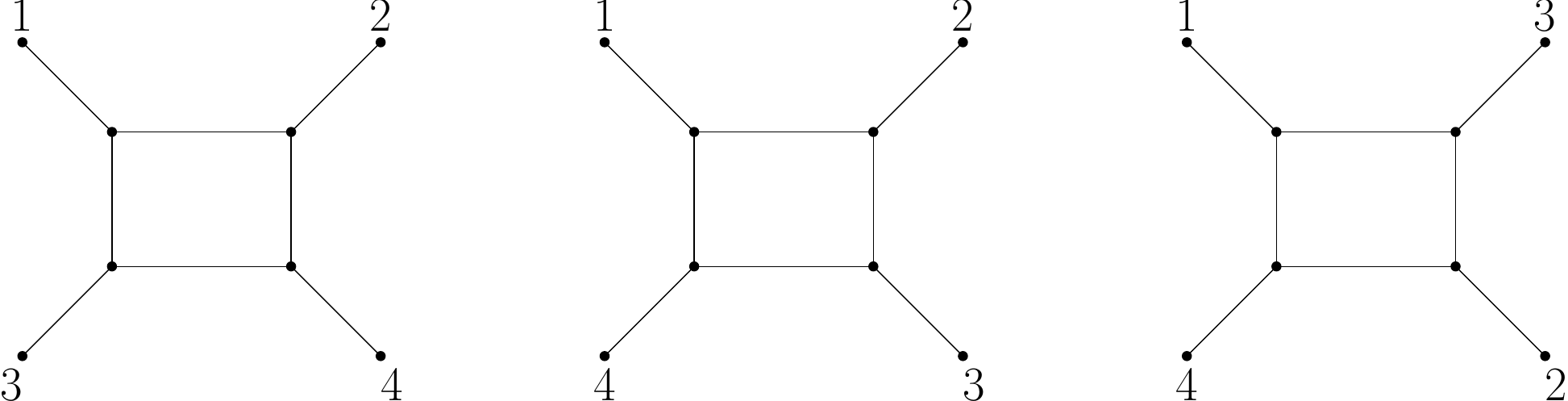}
\end{center}
\caption{\label{fig:OptGraphs4}Minimal realisations of
four-point metrics. The parallel sides of the middle rectangle have equal weight.}
\end{figure}

\item The case of $D_4$ is similar to that of $D_3$ in the sense that
there exists a single graph $\Gamma$ which, appropriately labelled and
weighted, realises any $X \in D_4$.  However, unlike for $D_3$, three
distinct labellings are required. The labelled graphs are depicted in
Figure~\ref{fig:OptGraphs4}.  For graphs in the relative interior of
$D_4$, the given realisation is optimal (and in fact the unique optimal
realisation). \hfill $\diamondsuit$

\end{enumerate}
\end{ex}

We now extend realisation of metric matrices by graphs to realisations
of arbitrary matrices in $\RR_{\geq 0}^{n \times n}$ with zeroes on the
diagonal.  For this we need an extension of the concept of a labelled
weighted graph. Let $i$ and $j$ be distinct elements of $[n]$. A {\em
detour} from $i$ to $j$ in an $[n]$-labelled weighted graph is simply
a walk $p$ starting at $\ell(i)$ and ending at $\ell(j)$ that has
larger total weight than the path of minimal weight between $\ell(i)$
and $\ell(j)$. Such a walk is allowed to traverse the same edge more
than once. The data specifying the detour is the triple $(i, j, p)$. A
\textit{labelled weighted graph with detours} is a tuple consisting of
a labelled weighted graph and a finite set of detours between distinct
ordered pairs $(i,j)$.

Let $(\Gamma, w, \ell, \mathcal{D})$ be an $[n]$-labelled weighted
graph with set of detours $\mathcal{D}$. It gives rise to a matrix
$A(\Gamma, w, \ell, \mathcal{D})$ whose entry at position $(i, j)$
equals the weight of the detour from $i$ to $j$, if there is any, or the
weight of a path of minimal weight between $i$ and $j$, if there is no
detour between $i$ and $j$ in $\mathcal{D}$. In particular, $A(\Gamma,
w, \ell, \mathcal{D})$ need not be symmetric, but its diagonal entries
are $0$. Again, if $X \in \RR_{\geq 0}^{n \times n}$ and $X = A(\Gamma,
w, \ell, \mathcal{D})$, then $(\Gamma, w, \ell, \mathcal{D})$ is said
to realise $X$. Any non-negative matrix with zeroes on the diagonal is
realised by some labelled weighted graph with detours. Observe also that
replacing all detours $(i,j,p)$ by the detours $(j,i,p')$, where $p'$
is the opposite of $p$, corresponds to transposing the realised matrix.

\begin{figure}
\begin{center}
% here the only two occurrences of subfigure; labels become 4a and 4b
\subfigure[\label{fig:DetourL3}Path with a single detour from $1$ to $2$.]%
{\includegraphics[scale=0.5]{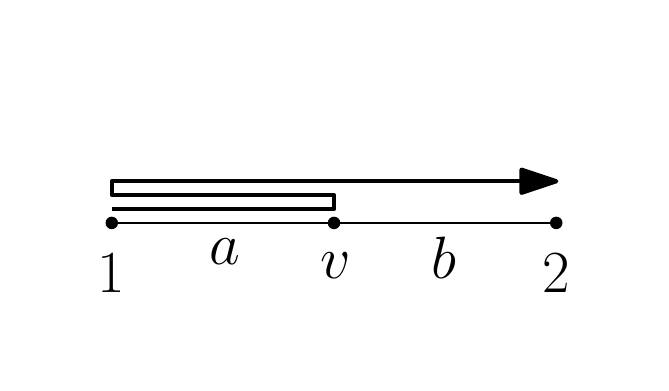}}
\quad \quad \quad \quad
\subfigure[\label{fig:DetourC4b}Graph with $4$ detours.]%
{\includegraphics[scale=0.5]{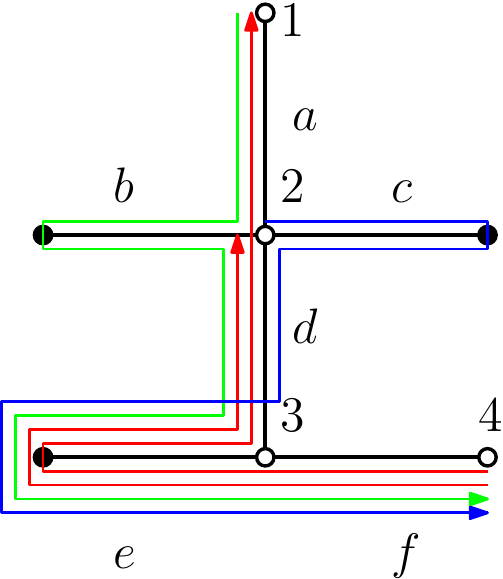}}
\caption{Examples of labelled weighted graphs with detours.}
\end{center}
\end{figure}

\begin{ex}
We give two examples of labelled weighted graphs with detours.  First,
the graph in Figure~\ref{fig:DetourL3} has a single detour from $1$
to $2$, and realises the matrix
\[ \begin{bmatrix} 0 & 3a+b \\ a+b & 0 \end{bmatrix}. \]
Except when $a=0$, this matrix is not in $G_2$. The example in
Figure~\ref{fig:DetourC4b} is more interesting. It has detours between
the ordered pairs $(1,4),(2,4),(4,1),(4,2)$. The weights $a,b,c,d,e,f$
are non-negative. By varying this six-tuple in $\RR_{\geq 0}^6$ this graph
with detours realises the $6$-dimensional cone of all matrices of the form
\[ 
A=
\begin{bmatrix}
0 & a & a+d & a+2b+d+2e+f\\
a & 0 & d & 2c+d+2e+f\\
a+d & d & 0 & f \\
f+2e+d+a & f+2e+d & f & 0
\end{bmatrix}. 
\]
Observe that both $(1,4)$ and $(4,1)$ are detours, and their lengths are
restricted by the inequality $a_{14} \geq a_{41}$ (indeed, the difference
equals $2b$). This $6$-dimensional cone is one of the maximal cones in
$G_4$, namely, cone $C_{10}$ in Figure~\ref{fig:cones_S4} below. The
graph-with-detours in Figure~\ref{fig:DetourC4b} represents these
inequalities in a visually attractive manner, but one also sees in one
glance that the cone of all matrices of the form is simplicial: it is
the image of $\RR_{\geq 0}^6$ under an injective linear transformation
into $\RR_{\geq 0}^{4 \times 4}$. This motivates our choice for
graphs-with-detours to represent cones of $G_3$ and, more importantly,
$G_4$. \hfill $\diamondsuit$
\end{ex}

By Lemma~\ref{lm:Kleene}, the Kleene star of a matrix $A$ in $G_n$ lies
in $\oD_n$. Thus it makes sense to look for a realisation of $A$ by a
labelled weighted graph with detours that, when forgetting the detours,
realises $A^*$. This is what we will do in the next two sections for $n=3$
and $n=4$.

\section{Three gossipers} \label{sec:Three}

Since $G_3$ is a pointed fan, no combinatorial information is lost
by intersecting that fan with a sphere centered around the all-zero
matrix. The resulting spherical polyhedral complex is depicted in
Figure \ref{fig:S3}. Detour graphs realising the maximal cones
can be constructed by realising the arrows in an arbitrary
manner as detours in the undirected graph. The middle cone is (the topological closure of) $D_3$, with
its three codimension-one faces corresponding to the second layer
in Figure~\ref{fig:OptPoset3} and its three codimension-two faces
corresponding to the third layer.

\begin{figure}[ht]
\begin{center}
\includegraphics[scale=0.5]{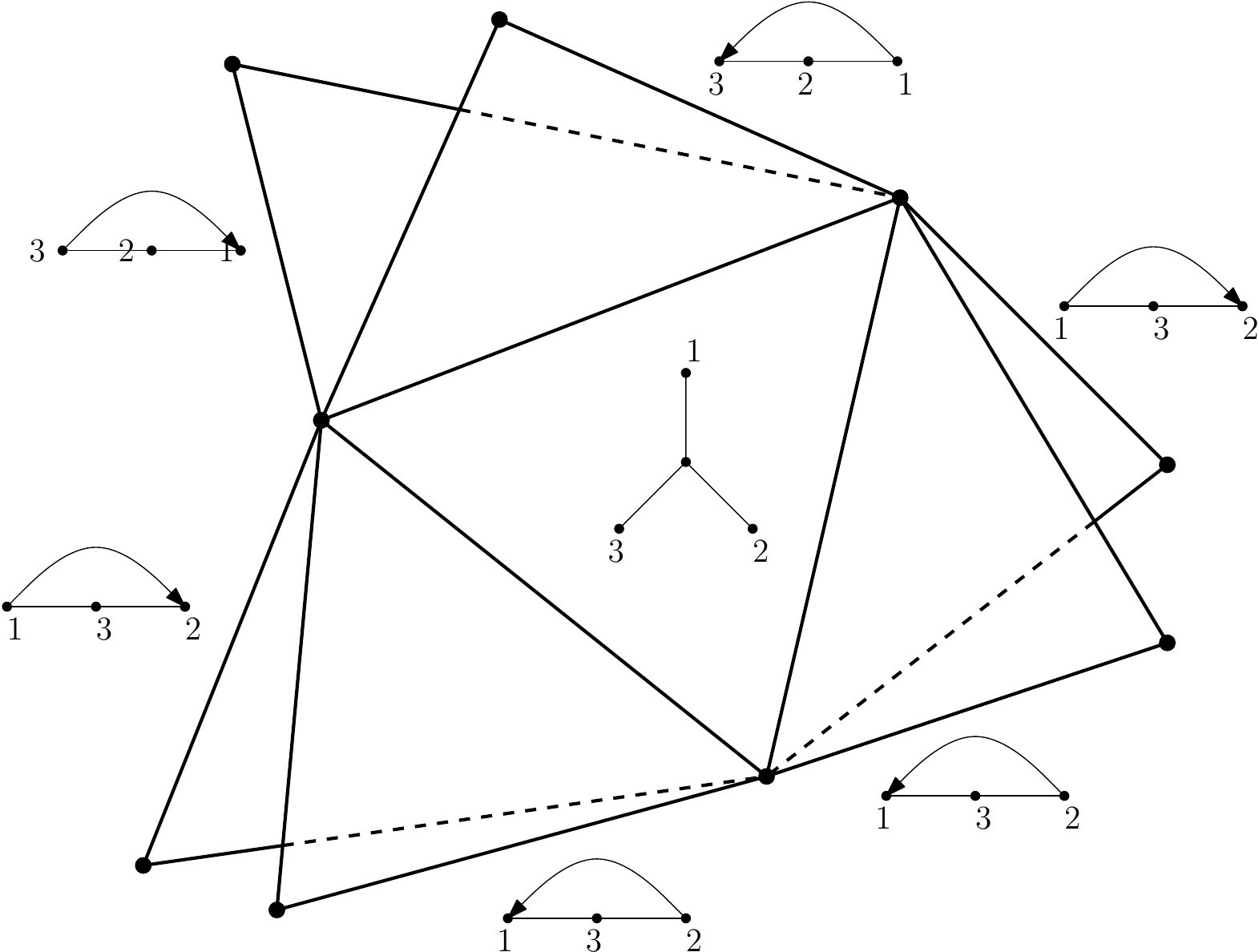}
\end{center}
\caption{\label{fig:S3}Representation of the spherical complex of
$G_3$. The labelled graphs with detours corresponding to the maximal
cells are indicated. The middle triangle represents the cone of distance
matrices; and on its codimension-one faces one of the three points ends
up between the other two points. The remaining codimension-one faces of
the remaining six cones are where one of the edge lengths in the Kleene
star becomes zero.}
\end{figure}

The computations to show that Figure \ref{fig:S3} gives all of $G_3$
are elementary and can be done by hand. We use pictorial notation and
write $A(\Gamma)$ for the matrix realised by a labelled weighted graph
with detours $\Gamma$. Here, instead of drawing a detour as
a walk, we draw it as an arrow whose length
is assumed to exceed the distance in the undirected graph. 
First, to prove that the matrices $A(\Gamma)$
with $\Gamma$ as in the figure are indeed in $G_3$ we observe that
\begin{equation}\label{exp:DecompS3}
A(\ \raisebox{-0.33\height}{\includegraphics[scale=0.4]{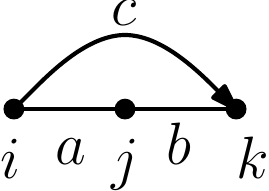}}\ )
=
C_{jk}(b)\, \odot \,
A(\ \raisebox{-0.40\height}{\includegraphics[scale=0.4]{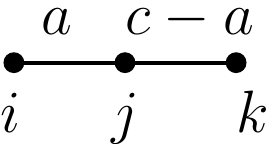}}\ )
=
A(\ \raisebox{-0.40\height}{\includegraphics[scale=0.4]{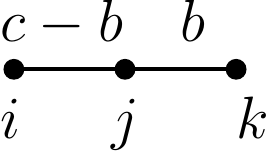}}\ )
\, \odot \,
C_{ij}(a),
\end{equation}
for any $c \geq a+b$ (and $a,b \geq 0$ as always). Together with the fact that
$C_{ij}(a) \odot C_{ij}(d)=C_{ij}(a \oplus d)$ this implies
that 
\[A(\ \raisebox{-0.33\height}{\includegraphics[scale=0.4]{figs/MCD}}\ )\, \odot\, C_{ij}(d)\mbox{ and }
C_{jk}(d) \, \odot\, A(\ \raisebox{-0.33\height}{\includegraphics[scale=0.4]{figs/MCD}}\ ) 
\]
are contained in the complex of Figure \ref{fig:S3} for all choices of 
$a, b, c$ and $d$ with $c \geq a+b$. Next we compute 
\begin{eqnarray*}
C_{ij}(d) \, \odot \,
A(\ \raisebox{-0.33\height}{\includegraphics[scale=0.4]{figs/MCD}}\ ),
& = &
\left\{\begin{array}{ll}
A(\ \raisebox{-0.33\height}{\includegraphics[scale=0.4]{figs/MCD}}\ ),
& c - b \leq d,\\[1.2em]
A(\ \raisebox{-0.33\height}{\includegraphics[scale=0.4]{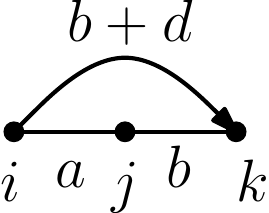}}\ ),
& a \leq d \leq c - b, \text{ and}\\[1.2em]
A(\ \raisebox{-0.33\height}{\includegraphics[scale=0.4]{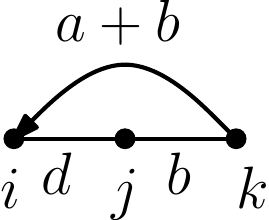}}\ ),
& 0 \leq d \leq a;
\end{array}\right.
\end{eqnarray*}
and, for $m := \max(a-b,b-a)$,
\begin{eqnarray*}
C_{ik}(d) \, \odot \,
A(\ \raisebox{-0.33\height}{\includegraphics[scale=0.4]{figs/MCD}}\ )
& = &
\left\{\begin{array}{ll}
A(\ \raisebox{-0.33\height}{\includegraphics[scale=0.4]{figs/MCD}}\ ),
& c \leq d, \\[1.2em]
A(\ \raisebox{-0.33\height}{\includegraphics[scale=0.4]{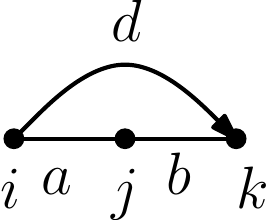}}\ ),
& a + b \leq d \leq c,\\[1.2em]
A(\ \raisebox{-0.4\height}{\includegraphics[scale=0.4]{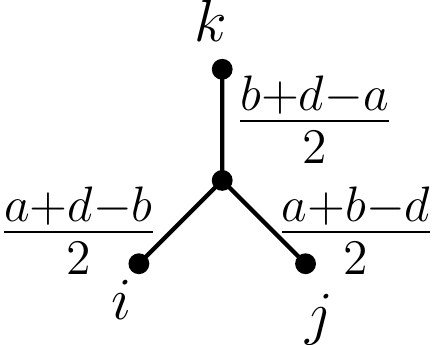}}\ ),
& m \leq d \leq a+b, \text{ and}\\[1.6em]
A(\ \raisebox{-0.33\height}{\includegraphics[scale=0.4]{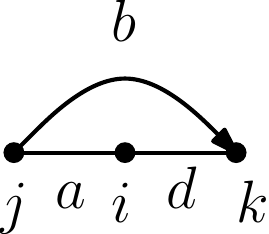}}\ ),
& 0 \leq d \leq m.
\end{array}\right.\\[1em]
\end{eqnarray*}
It follows by transposition that the products
\[
A(\ \raisebox{-0.33\height}{\includegraphics[scale=0.4]{figs/MCD}}\ )\, \odot\, C_{ik}(d),\mbox{ and }
A(\ \raisebox{-0.33\height}{\includegraphics[scale=0.4]{figs/MCD}}\ ) \odot\, C_{jk}(d)
\]
are also contained in one of the cones of Figure~\ref{fig:S3}.
This concludes the proof of Theorem~\ref{thm:n234} for $n=3$.

\section{Four gossipers} \label{sec:Four}

The computations for $G_4$ are too cumbersome to do by hand. Instead
we used {\tt Mathematica} to compute a fan structure on $G_4$. Figure
\ref{fig:cones_S4} gives realising graphs with detours of all the
cones of $G_4$, up to transposition and the action of $\Sym(4)$. The {\em
surplus length} of a detour from $i$ to $j$ is defined as the difference
between the length of the detour and the minimal distance between $i$
and $j$ in the graph. Two detours from $i$ to $j$ and from $k$ to $l$
have the same color if their surplus lengths are equal.

\begin{figure}[ht]
\begin{center}
\includegraphics[scale=0.5]{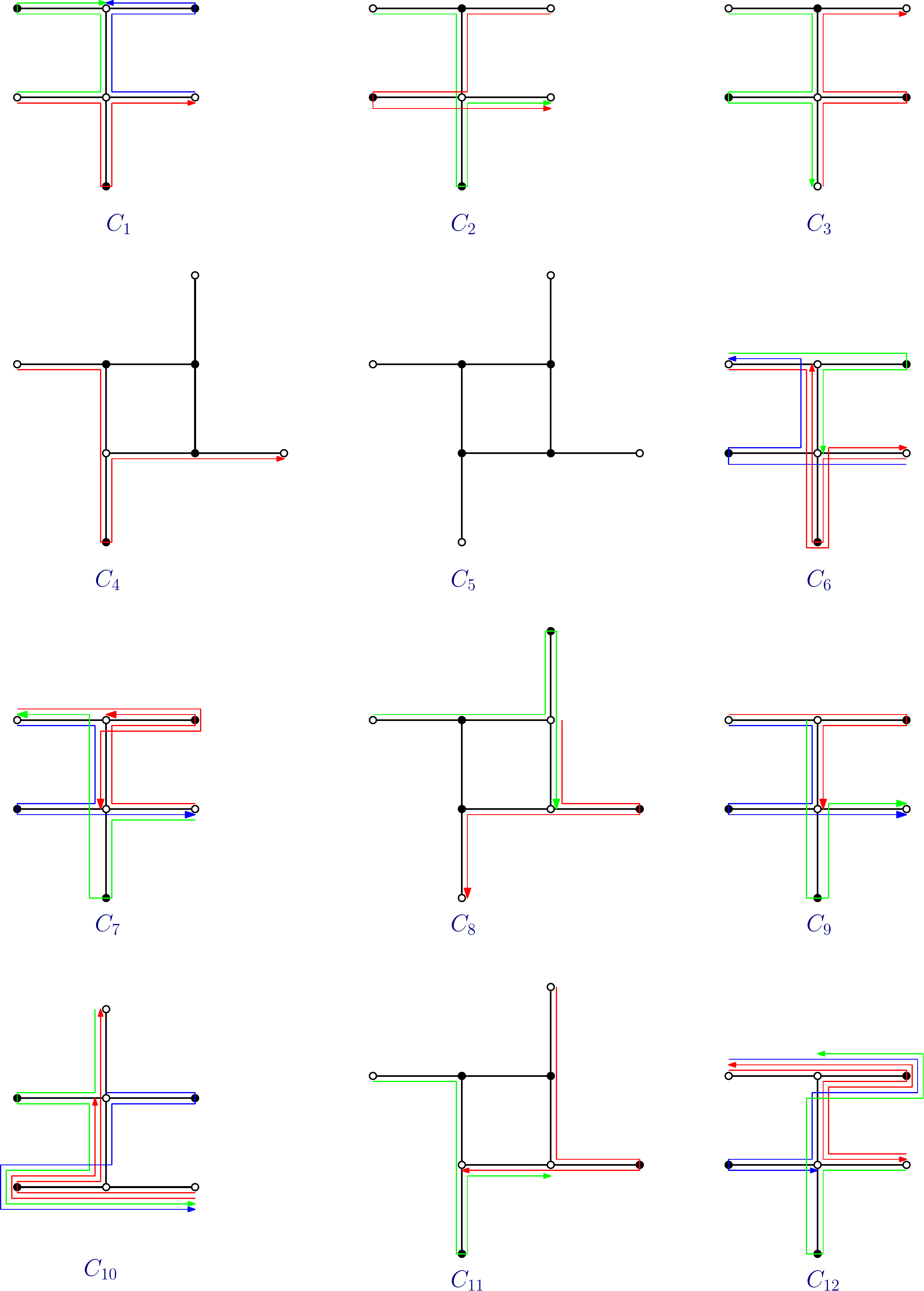}
\caption{\label{fig:cones_S4}Orbit representatives of labelled weighted
graphs with detours realising a polyhedral fan structure on $G_4$ with
simplicial cones. The white vertices are the labelled vertices.}
\end{center}
\end{figure}

\begin{figure}[ht]
\begin{center}
\includegraphics[scale=0.4]{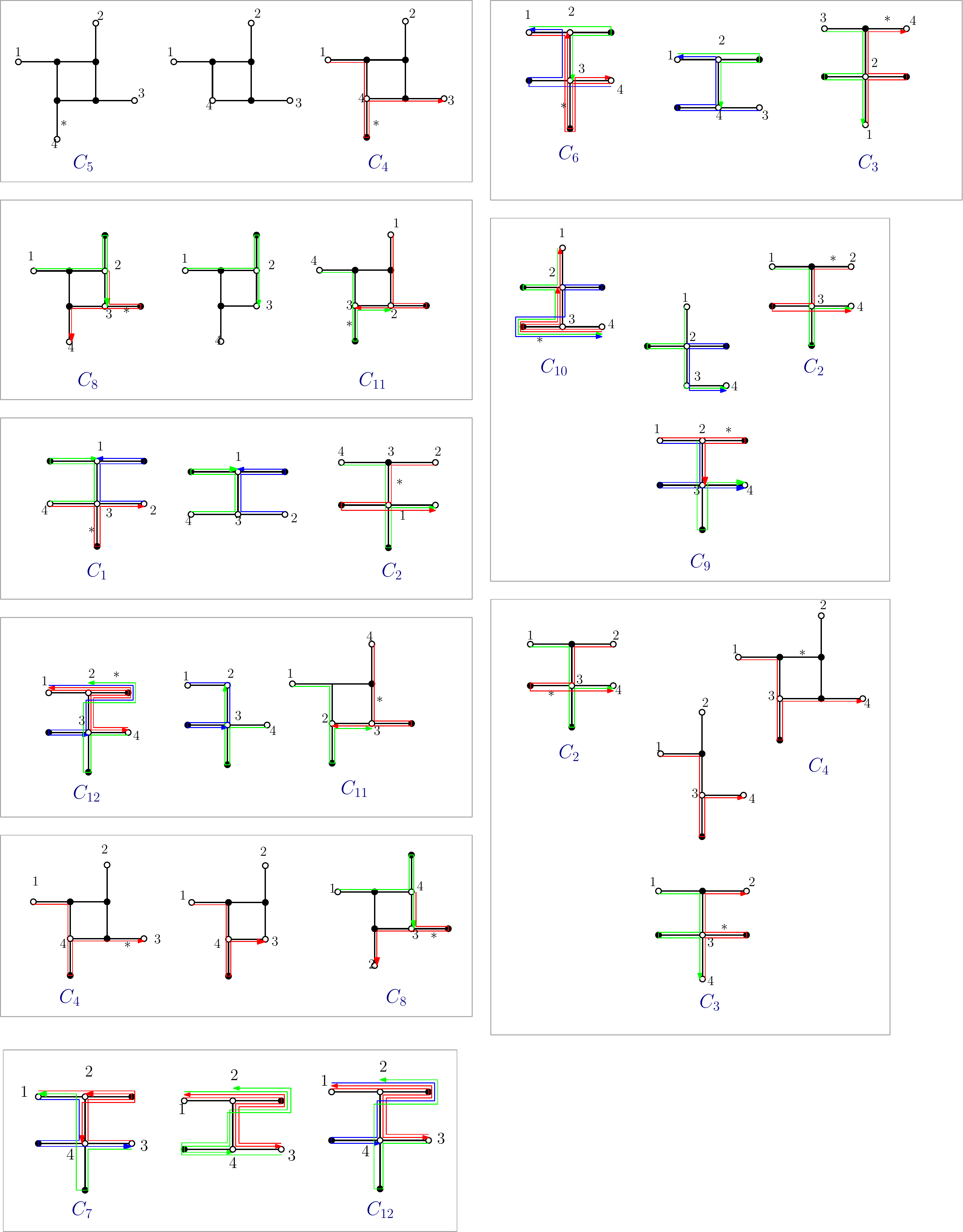}
\caption{\label{fig:connected}Walking from maximal cones to maximal
cones by edge contraction, except in case $C_7$--$C_{12}$. The edge to
be contracted is indicated by an asterix $*$. This shows that the cones
in the grey boxes intersect in a cone of dimension $5$. The intersection
between $C_7$ and $C_{12}$ is obtained by setting equal certain surplus
lengths in the graphs representing $C_7$ and $C_{12}$.}
\end{center}
\end{figure}

These graphs were obtained as follows. First, generate all $6^6$ possible
piecewise linear affine maps $[0,\infty]^6 \rightarrow G_4$ of the form
\[
(a_1, \ldots, a_6) \rightarrow C_{I_1}(a_1) \odot C_{I_2}(a_2) \odot \ldots \odot C_{I_6}(a_6),
\]
where $I_1,\ldots,I_6$ are unordered pairs of distinct indices. Among
the image cones, select only the six-dimensional ones, and compute
their linear spans. There are 289 different linear spans. Compute the
$\Sym(4)$-orbits on these spans; this yields $16$ orbits.  Choose a
representative for each of these orbits on spans, and for each
representative select all cones with that span. It turns out that, for
each representative span, one of the cones contains all other cones. To
show that the orbits of these 16 maximal cones give all of $G_4$,
left-multiply each of these 16 cones with all possible lossy phone call
matrices and show that the resulting unions of cones are contained in
the union of the 289 maximal cones; this is facilitated by the fact that
each of these cones is the intersection of $G_4$ with (the topological
closure in $[0,\infty]^{n \times n}$ of) a six-dimensional subspace.
Then we check that the faces of these 289 six-dimensional cones do
indeed form a polyhedral fan, i.e., that the intersection of any two
of these faces is a common face of both. In the process of this check,
which we performed both with \verb+Mathematica+ and (more rapidly)
with \verb+polymake+ \cite{polymake}, we find that the fan has $f$-vector
$(43,327,1042,1560,1092,289)$. This latter check yields the statement
about the unique coarsest fan structure in Theorem~\ref{thm:n234}.

Next, the group $\ZZ/2\ZZ$ acts on $G_4$ by transposition. Taking orbit
representatives under the larger group $\Sym(4) \times (\ZZ/2\ZZ)$ from
among the $16$ yields 11 cones. Among these, 9 are simplicial (have
six facets), the cone $D_4$ has 12 facets, and the remaining cone has
9 facets. The cone $D_4$ is the union of three simplicial cones (see
Figure~\ref{fig:OptGraphs4}), which are permuted by $\Sym(4)$, so we need
only one. This is $C_5$ in Figure~\ref{fig:cones_S4}. The cone with $9$
facets turns out to be the union of two simplicial cones. Splitting this
up yields $C_{11}$ and $C_{12}$ in the figure. It turns out
that each $C_i$ is the image of $\RR_{\geq 0}^6$ under a linear map into $\RR_{\geq
0}^{4 \times 4}$ with non-negative integral entries with respect to
the standard bases, and that these maps can be realised using weighted,
labelled graphs with detours. These are the graphs in the picture. The
graphs without the detours realise the Kleene star $A^*$ with $A \in C_i$.

Finally, connectivity in codimension $1$ is proved by Figure
\ref{fig:connected}. It shows that any maximal cone can be connected to
$D_4$ by passing through (relatively open) codimension-one faces; note the
specified labelling. Most intersections in Figure \ref{fig:connected} are
of a simple type, where one of the edge weights becomes zero to go from
one cone to the neighbouring cone; these contracted edges are then marked
with an asterix on both sides.  The only exception is the connection from
$C_7$ to $C_{12}$. Although (suitable elements in the $\Sym(4)$-orbits of)
these cones intersect in a five-dimensional boundary cone, the boundary
cone is obtained from the parametrizations specified by the graphs with
detours by restricting the parametrization to a hyperplane where two of
the weights are equal.  This leads to the following theorem.

\begin{thm}
The cones realised by the graphs of Figure
\ref{fig:cones_S4} give a polyhedral fan structure on $G_4$.
This polyhedral fan is pure of dimension $6$ and connected
in codimension $1$. Its intersection with a sphere around
the origin is a simplicial spherical complex. Moreover,
every element of $G_4$ is the product of (at most) $6$ lossy
phone call matrices.
\end{thm}

\begin{re}
The spherical complex of Figure~\ref{fig:S3} clearly has trivial
homology. This phenomenon persists for $n=4$: a computation using
\verb+polymake+ shows that all homology groups of the intersection of
$G_4$ with the unit sphere in $16$-dimensional space are zero. We do
not know whether this is true for general $n$.
\end{re}

\section{Five gossipers} \label{sec:Five}
More extensive computations establish the claimed facts about $G_5$.
Since $6^6$ is a small number, but $10^{10}$ is not,
the computation requires many refinements. We omit the details.
It turns out that every element has an expression as a tropical product
of at most 10 lossy phone call matrices.  The set $G_5$ is the support
of a polyhedral fan which is pure of dimension 10, and connected in
codimension 1.
Some statistics are given in Table~\ref{table:stats}. The single orbit
of size 1 is that of $D_n$.

\begin{table}[ht]
\begin{center}
\begin{tabular}{cccl}
$n$ & \# spans & \# orbits & orbit size distribution \\
\hline
2 & 1 & 1 & 1\tim1 \\
3 & 7 & 2 & 1\tim1, 1\tim6 \\
4 & 289 & 16 & 1\tim1, 6\tim12, 9\tim24 \\
5 & 91151 & 787 & 1\tim1, 2\tim20, 1\tim30, 48\tim60, 735\tim120
\end{tabular}
\medskip
\end{center}
\caption{\label{table:stats}Numbers of subspaces spanned by full-dimensional
cones, and their numbers of orbits under ${\rm Sym}(n)$.}
\end{table}
The situation for $n = 5$ is more complicated than that for
smaller $n$ in that it is no longer true that the subspace spanned by a
polyhedral cone of maximal dimension intersects $G_5$ in a convex cone
(recall that for $n=4$ this did hold, and that we used this in the proof
that $G_4$ has a unique coarsest fan structure).

\begin{ex}
Consider the open, $10$-dimensional cone $P$ consisting of all matrices 
$$
C_{45}(a)C_{34}(b)C_{45}(c)C_{24}(d)C_{45}(e)
C_{14}(f)C_{12}(g)C_{23}(h)C_{13}(i)C_{15}(j)
$$
where we have left out the $\odot$ sign for brevity, and 
where the parameters $a,\ldots,j$ satisfy the inequalities
\begin{gather*}
a > c,~~  e > c,~~  f > d+g,~~  b+d > h,~~  h > g+i,~~  c+d+g+i > a+b,\\
b+i > d+g,~~  c+d+g > j,~~  i+j > b+c,~~  c+j > d+g,~~  g+j > d+e.
\end{gather*}
Similarly, consider the open, $10$-dimensional cone $Q$ consisting of all matrices
$$
C_{45}(a)C_{34}(b)C_{45}(c)C_{24}(d)C_{45}(e)
C_{15}(f)C_{12}(g)C_{24}(h)C_{23}(i)C_{13}(j)
$$
with inequalities
\begin{gather*}
a > c,~~  e > c,~~  c+d+g > f,~~  c+f > d+g,~~  f+g > c+d,\\
h > d,~~  b+d > i,~~  i > g+j,~~  f+j > a+b,~~  b+j > d+g.
\end{gather*}
In matrix form, these matrices are 
$$
%\left[\begin{array}{ccccc}
\begin{bmatrix}
    0       & g       & i       & f       & j \\
    g       & 0       & g+i     & d       & d+e \\
    i       & h       & 0       & b       & b+c \\
    d+g     & d       & b       & 0       & c \\
    j       & c+d     & a+b     & c       & 0
\end{bmatrix}
%\end{array}\right]
\text{ and }
\begin{bmatrix}
%\left[\begin{array}{ccccc}
    0       & g       & j       & g+h     & f \\
    g       & 0       & g+j     & d       & d+e \\
    j       & i       & 0       & b       & b+c \\
    d+g     & d       & b       & 0       & c \\
    f       & c+d     & a+b     & c       & 0
\end{bmatrix}.
%\end{array}\right].
$$
The linear spans of $P$ and $Q$ are the same, and the closures of $P$
and $Q$ cover all $10$-dimensional cones in $G_5$ with this span.
The latter matrix becomes the former after the substitution
$f \to j$,~ $h \to f-g$,~ $i \to h$,~ $j \to i$, and this substitution
turns its inequalities into
\begin{gather*}
a > c,~~  e > c,~~  c+d+g > j,~~  c+j > d+g,~~  g+j > c+d,\\
f > d+g,~~  b+d > h,~~  h > g+i,~~  i+j > a+b,~~  b+i > d+g.
\end{gather*}
We see that both cones satisfy
\begin{gather*}
a > c,~~  e > c,~~  f > d+g,~~  b+d > h,~~  h > g+i,~~ b+i > d+g,~~
c+j > d+g,\\
c+d+g > j,~~  i+j > b+c,~~ c+d+g+i > a+b,~~  g+j > c+d.
\end{gather*}
In addition, $P$ satisfies $g+j > d+e$ and $Q$ satisfies
$i+j > a+b$. It follows that $P$ and $Q$ have a 
10-dimensional intersection and their union is not a convex cone.
\hfill $\diamondsuit$
\end{ex}

%To give interesting, more detailed information about the polyhedral
%structure of $G_n$ for $n \geq 5$ we would first want to have a natural
%polyhedral supporting $G_n$. This might require giving a combinatorial
%proof of Theorem~\ref{thm:main}; and such a proof eludes us at
%present.

\section{Tropicalising matrix groups} \label{sec:Proof}

In the previous sections we have established Theorem~\ref{thm:n234}
through explicit computations. We do not know of any systematic,
combinatorial description of a polyhedral structure on $G_n$ for larger
$n$. However, we will now establish that $G_n$, which is the support
set of {\em some} finite polyhedral fan by Lemma~\ref{lm:UpperBound},
has dimension $\binom{n}{2}$. Clearly, since $G_n$ contains $D_n$,
we have $\dim G_n \geq \dim D_n = \binom{n}{2}$.  So the difficulty of
Theorem~\ref{thm:main} is in proving that its dimension does not exceed
$\binom{n}{2}$.

For this, we make an excursion into tropical geometry.  Recall that if
$K$ is a field with a non-Archimedean valuation $v:K \to \Rb:=\RR \cup
\{\infty\}$ and if $I \subseteq K[x_1,\ldots,x_m]$ is an ideal, then
the tropical variety associated to $I$ is the set of all $w \in \Rb^n$
such that for each polynomial $f = \sum_\alpha c_\alpha x^\alpha \in I$
the minimum $\min_\alpha(v(c_\alpha)+w \cdot \alpha)$ is attained for at
least two distinct $\alpha \in \NN^n$. We denote this tropicalisation
by $\Trop(X)$, where $X$ is the scheme over $K$ defined by $I$. For
standard tropical notions we refer to \cite{Maclagan15}. If $(L,v)$
is any valued extension of $K$, then the coordinate-wise valuation map
$v:L^n \to \Rb^n$ maps $X(L)$ into $\Trop(X)$. If, moreover, $v:L \to \Rb$
is non-trivial and $L$ is algebraically closed, then the image of the
map $X(L) \to \Trop(X)$ is dense in $\Trop(X)$ in the Euclidean topology.
Together with the Bieri-Groves theorem~\cite{Bieri84}, this implies
that the set $\Trop(X)$ is (the closure in $\Rb^n$ of) a polyhedral complex of
dimension equal to $\dim X$.

We now specialise to matrix groups. As a warm-up, consider the
special linear group $\lieg{SL}_n$, defined over $\QQ$ by the
single polynomial $\det(x)-1$ where $x$ is an $n \times n$-matrix
of indeterminates. Since there is only one defining polynomial and all
its coefficients 
are $\pm 1$, the valuation does not matter and $\Trop(\lieg{SL}_n)$
equals the set of all $A=(a_{ij})_{ij} \in \Rb^{n \times n}$ for which
the {\em tropical determinant} \[ \tdet(A):=\min_{\pi \in \Sym(n)}
(a_{1\pi(1)}+\cdots+a_{n\pi(n)}) \] is either zero, or else negative
and attained at least twice.

\begin{prop}
The tropicalisation $\Trop(\lieg{SL}_n)$ is a monoid under tropical
matrix multiplication.
\end{prop}

\begin{proof}
For $A,B \in \Trop(\lieg{SL}_n)$ set $C:=A \odot B$. A straightforward
computation shows that the tropical determinant is (tropically)
submultiplicative, so that $\tdet(C) \leq \tdet(A) + \tdet(B) \leq 0 +
0 = 0$. Hence it suffices to show that if $\tdet(C)<0$, then there are
at least two permutations realising the minimum in the definition of
$\tdet(C)$. Let $\pi \in \Sym(n)$ be one minimiser of the expression
$c_{1\pi(1)}+\cdots+c_{n\pi(n)}$. For each $i \in [n]$ let $\sigma(i)\in
[n]$ be such that
$c_{i\pi(i)}=a_{i\sigma(i)}+b_{\sigma(i)\pi(i)}$. Now
there are two cases: either $\sigma$ is a permutation, or there exist
$i,j$ with $\sigma(i)=\sigma(j)$. In the latter case, also the permutation
$\pi \circ (i,j) \neq \pi$ is a minimiser, and we are done. In the former case,
write $\pi=\tau \circ \sigma$. Then we have have
\[ 0> \tdet(C)=c_{1\pi(1)}+\cdots+c_{n\pi(n)}=
(a_{1\sigma(1)}+\cdots+a_{n\sigma(n)})+
(b_{1\tau(1)}+\cdots+b_{n\tau(n)}), \]
so that at least one of $\tdet(A)$ and $\tdet(B)$ is negative. If
$\tdet(A)<0$, then since $A \in \Trop(\lieg{SL}_n)$, there exists a permutation
$\sigma' \neq \sigma$ such that
\[ a_{1\sigma'(1)}+\cdots+a_{n\sigma'(n)} \leq a_{1
\sigma(n)}+\cdots+a_{n \sigma(n)}, \]
and we find that $\pi':=\tau \circ \sigma' \neq \pi$ is another
minimiser. The argument for $\tdet(B)<0$ is similar.
\end{proof}

In general, it is not true that the tropicalisation of a matrix group
(relative to the standard coordinates) is a monoid under tropical
multiplication.

\begin{ex}
Let $G$ denote the group of $4 \times 4$-matrices of the form
\[ 
\begin{bmatrix}
1 & x & -x & 0 \\
0 & 1 & 0  & x \\
0 & 0 & 1  & x \\
0 & 0 & 0  & 1
\end{bmatrix}
\]
where $x$ runs through the field $K$. This is a one-dimensional algebraic
group isomorphic to the additive group, whose
tropicalisation consists of all matrices
\[ 
\begin{bmatrix}
0 & a & a & \infty \\
\infty & 0 & \infty  & a \\
\infty & \infty & 0  & a \\
\infty & \infty & \infty  & 0
\end{bmatrix}
\]
where $a \in \Rb$. But we have 
\[
\begin{bmatrix}
0 & a & a & \infty \\
\infty & 0 & \infty  & a \\
\infty & \infty & 0  & a \\
\infty & \infty & \infty  & 0
\end{bmatrix}
\odot 
\begin{bmatrix}
0 & b & b & \infty \\
\infty & 0 & \infty  & b \\
\infty & \infty & 0  & b \\
\infty & \infty & \infty  & 0
\end{bmatrix}
=
\begin{bmatrix}
0 & \!\min\{a,b\}  & \!\min\{a,b\} & a+b \\
\infty & 0 & \infty  & \!\min\{a,b\} \\
\infty & \infty & 0  & \!\min\{a,b\} \\
\infty & \infty & \infty  & 0
\end{bmatrix}\!, 
\]
which for $a,b<\infty$ does not lie in $\Trop(G)$. \hfill
$\diamondsuit$
\end{ex}

Now consider the orthogonal group $\lieg{O}_n$ consisting of all matrices
$g$ that satisfy $g^\tp g=I$. We do not know whether $\Trop(\lieg{O}_n)$
is a monoid under tropical matrix multiplication, but we shall see that
this tropicalisation does contain the lossy gossip monoid.  For this, we
take $L$ to be the field $\CC\{\{t\}\}$ of Puiseux series in a variable
$t$, and $v$ to be the order of a Puiseux series at $0$.
Motivated by the analogy between the lossy
phone call matrices $C_{ij}(a), a \in \RR_{\geq 0}$ and one-parameter
subgroups of algebraic groups (see Section~\ref{sec:Prel}), we introduce
the one-parameter subgroups $g_{ij}(x)$ of $\lieg{O}_n$ by
\[ 
g_{ij}(x):=
\begin{bmatrix} 
	1 &         &        &        &\\
	  & \cos(x) & \cdots &-\sin(x)&\\
	  & \vdots  & 1       & \vdots&        &\\
	  & \sin(x) & \cdots & \cos(x) &\\
	  &         &        &        & 1 
\end{bmatrix}, 
\]
where the $1$s stand for identity matrices, the cosines and sines are
in the $\{i,j\} \times \{i,j\}$-submatrix, and the empty entries are
$0$. For any choice of $x$ in the field $L$ whose order $\val(x)$ at
zero is positive, the matrix $g_{ij}(x)$ is a well-defined matrix in
the orthogonal group $\lieg{O}_n(L)$.

\begin{prop} \label{prop:GnTrop}
The lossy gossip monoid $G_n$ is contained in $\Trop(\lieg{O}_n)$.
\end{prop}

\begin{proof}
First note that $\val(g_{ij}(x))=C_{ij}(\val(x))$, so the statement
would be immediate if we knew that $\Trop(\lieg{O}_n)$ 
were closed under tropical matrix multiplication. We prove something
weaker. Let $a_1,\ldots,a_k$ be strictly positive rational numbers and
let $(i_1,j_1),\ldots,(i_k,j_k)$ be pairs of distinct
indices. Then for a vector $(c_1,\ldots,c_k) \in \CC^k$
outside some proper hypersurface, no cancellation takes
place in the expression 
\[ g_{i_1,j_1}(c_1 t^{a_1}) \cdots g_{i_k,j_k}(c_k t^{a_k}),
\]
in the sense that 
\[ \val[g_{i_1,j_1}(c_1 t^{a_1}) \cdots g_{i_k,j_k}(c_k
t^{a_k})]
= \val[g_{i_1,j_1}(c_1 t^{a_1})] \odot \cdots \odot
\val[g_{i_k,j_k}(c_k t^{a_k})]. \] 
Here the right-hand side equals $C_{i_1,j_1}(a_1) \odot
\cdots \odot C_{i_k,j_k}(a_k)$, and lies in
$\Trop(O_n)$ since the left-hand side does. Since $\Trop(O_n)$ is closed in the
Euclidean topology, all of $G_n$ is contained in it.
\end{proof}

The dimension claim in Theorem~\ref{thm:main} follows from
Proposition~\ref{prop:GnTrop}, the Bieri-Groves theorem, and the fact
that $\dim \lieg{O}_n=\binom{n}{2}$.

For $n=1,2,3$, we can say a little bit more about
$\Trop(\lieg{O}_n)$. 

\begin{ex}
For $n=1$, $\Trop(\lieg{O}_n)$ consists of the single $1 \times 1$-matrix $0$.
Next, for a $2 \times 2$-matrix 
\[ g=\begin{bmatrix} x & y \\ u & v \end{bmatrix} \text{
with valuation } \begin{bmatrix} a & b \\ c & d \end{bmatrix} \]
to lie in $\lieg{O}_2$ we need that $x^2 + u^2-1=y^2+v^2-1=0=xy+uv=0$, and these equations generate the ideal
of $\lieg{O}_2$.  Tropicalising these equations yields that
$\min\{a,c,0\},\min\{b,d,0\},\min\{a+b,c+d\}$ are all attained at least
twice. This is not sufficient to characterise $\Trop(\lieg{O}_2)$;
indeed, for any negative $a,b$ the matrix
\[ \begin{bmatrix} a & b \\ a & b \end{bmatrix} \]
satisfies all tropical equations above, but (unless $a=b$) not
the tropicalisation of the equation $xv-yu=1$ which expresses that
$\lieg{O}_2 \subseteq \lieg{SL}_2$. Imposing this additional condition,
i.e., that $\min\{a+d,b+c,0\}$ is attained at least twice, we find that
$\Trop(\lieg{O}_2)$ consists of three cones:
\begin{align*} \Trop(\lieg{O}_2)&= 
\left\{\begin{bmatrix} 0 & a \\ a & 0 \end{bmatrix} \mid a \in [0,\infty]\right\}
\cup 
\left\{\begin{bmatrix} a & 0 \\ 0 & a \end{bmatrix} \mid a
\in [0,\infty]\right\}\\
&\cup 
\left\{\begin{bmatrix} a & a \\ a & a \end{bmatrix} \mid a
\in (-\infty,0]\right\}.
\end{align*} 
The first cone is $G_2$, the second cone is $G_2$ with the columns
reversed, and the third cone makes the fan balanced. 
%A straightforward
%computation shows that $\Trop(O_2)$ is a monoid with respect to tropical
%matrix multiplication.

In general, if a variety is stable under a coordinate permutation,
then its tropicalisation is stable under the same coordinate
permutation. Consequently, $\Trop(O_n)$ is stable under permuting rows,
under permuting columns, and under matrix transposition.

For $n=3$, a computation using \verb+gfan+ \cite{gfan} shows that
the quadratic equations expressing that columns and rows both form
orthonormal bases, together with the equation $\det-1$, do {\em not}
form a tropical basis. For example, the four-dimensional cone of
matrices 
\[ 
\begin{bmatrix}
a & a & b \\
a & a & b \\
c & c & d 
\end{bmatrix}
\]
with $a \leq b \leq c \leq 0 \leq d$ is contained in the 
tropical prevariety defined by the corresponding tropical
equations, and for dimension reasons cannot belong to the
three-dimensional fan $\Trop(\lieg{O}_3)$.  

However, these quadratic equations {\em do} suffice to prove that
$\Trop(\lieg{O}_3) \cap [0,\infty]^{3 \times 3}$ is equal to $\Sym(3)
\cdot G_3$, i.e., obtained from $G_3$ by permuting rows. Indeed, let
a $3 \times 3$-matrix $A$ in $[0,\infty]^{3 \times 3}$ satisfy the
tropicalisations of these equations. Then $\tdet(A)=0$, hence after
permuting rows $A$ has zeroes on the diagonal. Now we distinguish two
cases. First, assume that $A$ is symmetric:
\[ A=\begin{bmatrix}
                0 & a & b \\
                a & 0 & c \\
                b & c & 0
\end{bmatrix}. \]
Then we claim that $A$ lies in $D_3$. Indeed, suppose that $a > b+c$.
Then the tropicalisation of the condition that the first two columns are
perpendicular does not hold for $A$. Hence $a \leq b+c$ and similarly
for the other triangle inequalities; we conclude that $A \in D_3$. Next, assume that $A$ is not
symmetric. After conjugation
with a permutation matrix, we may assume
that $A$ is of the form
\[ \begin{bmatrix}
                0 & a & b \\
                d & 0 & c \\
                e & f & 0
\end{bmatrix} \]
with $a>d$. Then the tropical perpendicularity of the first two
columns yields $d=e+f$, that of the last two columns yields
$c=f$, and that of the first two rows yields $d=b+c$. So $A$ looks like
\[ \begin{bmatrix}
                0 & a & b \\
                b+c & 0 & c \\
                b & c & 0
\end{bmatrix}, \]
which is one of the cones in $G_3$. 
\hfill $\diamondsuit$
\end{ex} 

\begin{re}
We do not know whether the equality $\Sym(n) \cdot G_n=\Trop(\lieg{O}_n)
\cap [0,\infty]^{n \times n}$ (where the action of $\Sym(n)$ is by left
multiplication) holds for all $n$. If true, then this would
be interesting from the perspective of algebraic groups over
non-Archimedean fields: it would say that the image under $v$ of the
compact subgroup $\lieg{O}_n(L^0) \subseteq \lieg{O}_n(L)$, where
$L^0$ is the valuation ring of $L$, is (dense in) the lossy gossip
monoid.  But we see no reason to believe that this is true in general.
A computational hurdle to checking this even for $n=4$ is the computation
of a polyhedral fan supporting $\Trop(\lieg{O}_n)$. For $n=3$ this can
still be done using \verb+gfan+, and it results in a fan with $f$-vector
$(580,1698,1143)$. Among the $1143$ three-dimensional cones, $1008$ are
contained in the positive orthant, as opposed to the $6 \cdot 7=42$ found
by applying row permutations to the cones in $G_3$. This suggests that
\verb+gfan+ does not automatically find the most efficient fan structure
on $\lieg{O}_n$, and at present we do not know how to overcome this.
\end{re}

\section{Ordinary gossip} \label{sec:Ordinarygossip}

In this section we study the {\em ordinary gossip monoid}
$G_n(\{0,\infty\})$, which is the submonoid of $G_n$ of matrices with
entries in $\{0,\infty\}$. Note that there is a surjective homorphism $G_n
\to G_n(\{0,\infty\})$ mapping non-$\infty$ entries to $0$ and $\infty$ to
$\infty$, which shows that the length of an element of $G_n(\{0,\infty\})$
inside $G_n$ is the same as the minimal number of non-lossy phone
calls $C_{ij}(0)$ needed to express it. A classical result says that
length of the all-zero matrix is exactly $1$ for $n=2$, $3$ for $n=3$,
and $2n-4$ for $n \geq 4$ \cite{Baker72,Bumby81,Hajnal72,Tijdeman71},
and this result spurred a lot of further activity on gossip networks.
But the all-zero matrix does not necessarily have the largest possible
length---see Table~\ref{table:snbb}, which records sizes and maximal
element lengths for $G_n(\{0,\infty\})$ with $n \leq 9$. The first $8$
rows were computed by former Eindhoven Master's student Jochem Berndsen
\cite{Berndsen12}. 

\begin{table}[ht]
\begin{center}
\begin{tabular}{c|r|r}
$n$ & $|G_n(\{0,\infty\})|$ & max. length \\
\hline
1 & 1 & 0 \\
2 & 2 & 1 \\
3 & 11 & 3 \\
4 & 189 & 4 \\
5 & 9152 & 6 \\
6 & 1,092,473 & 10 \\
7 & 293,656,554 & 13 \\
8 & 166,244,338,221 & 16\\
9 & 188,620,758,836,916 & 19\\
\hline
\end{tabular}
\end{center}
\caption{\label{table:snbb}Sizes and maximal lengths of
$G_n(\{0,\infty\})$, for $n = 1, \ldots, 9$.}
\end{table}

While we do not know the maximal length of an element in
$G_n(\{0,\infty\})$ for general $n$, we do have an upper bound, namely,
the maximal number of factors in an irredundant product. This number,
in turn, is bounded from above by $\binom{n}{2}$, as we now prove.

\begin{proof}[Proof of Theorem~\ref{thm:pessimal} and % of
Corollary~\ref{cor:irredundant}.]
Consider $n$ gossipers, initially each with a different gossip item
unknown to all other gossipers. They communicate by telephone, and
whenever two gossipers talk, each tells the other all he knows. We will
determine the maximal length of a sequence of calls, when in each call
at least one participant learns something new. The answer turns out to
be $\binom{n}{2}$.

That $\binom{n}{2}$ is a lower bound, is shown by the following scenario:
Number the gossipers $1,\ldots,n$. All calls involve gossiper $1$.
For $i=2,\ldots,n$ he calls $i,i-1,\ldots,2$, for a total of
$1+2+\cdots+(n-1) = \binom{n}{2}$ calls. There are many other scenarios
attaining $\binom{n}{2}$, and it does not seem easy to classify them.

We now argue that $\binom{n}{2}$ is an upper bound. Although we will
not use this, we remark that it is easy to see that $2 \cdot
\binom{n}{2}=n(n-1)$ is an upper bound. After all, each of
the $n$ participants must learn $n-1$ items, and in each call at least
one participant learns something.

Let $I_1,I_2,\ldots,I_\ell$ be a sequence of unordered pairs from $[n]$
representing phone calls where in each call at least one participant
learns something new. To each $I_a$ we associate the homomorphism
$\phi_a:=\lieg{SO}_2(\CC) \to \lieg{SO}_n(\CC)$ that maps a $2 \times
2$ matrix $g$ to the matrix that has $g$ in the $I_a \times I_a$-block
and otherwise has zeroes outside the diagonal and ones on the diagonal.
For each $k \leq \ell$ we obtain a morphism of varieties (not a group
homomorphism) $\psi_k:\lieg{SO}_2(\CC)^k \to \lieg{SO}_n(\CC)$ sending
$(g_1,\ldots,g_k)$ to $\phi_1(g_1) \cdots \phi_k(g_k)$. Let $X_k$ be the
closure of the image of $\psi_k$; this is an irreducible subvariety of
$\lieg{SO}_n(\CC)$. The $(i,j)$-matrix entry is identically zero on $X_k$
if and only if gossiper $j$ does not know gossip $i$ after the first $k$
phone calls. Since some gossiper learns something new in the $k$-th
phone call, some matrix
entry is identically zero on $X_{k-1}$ which is not identically zero
on $X_k$. Consequently, we have $0=\dim X_0<\dim X_1<\ldots<\dim
X_\ell$. But all $X_k$ are contained in the variety
$\lieg{SO}_n(\CC)$ of dimension $\binom{n}{2}$, so we conclude that
$\ell \leq \binom{n}{2}$. 

This concludes the proof of
Theorem~\ref{thm:pessimal}. Corollary~\ref{cor:irredundant} follows
because in any irredundant product of phone calls, every initial segment
must be a sequence of phone calls in each of which at least one party
learns something new.
\end{proof}

We computed the longest irredundant products of phone calls
for small $n$, see Table~\ref{table:irred}.

\begin{table}[ht]
\begin{center}
\begin{tabular}{c|ccccccccc}
$n$ & 1 & 2 & 3 & 4 & 5 & 6 & 7 & 8 \\
\hline
$l_n$ & 0 & 1 & 3 & 5 & 8 & 12 & 16 & $\geq 21$
\end{tabular}
\medskip
\end{center}
\caption{\label{table:irred}
Maximum length $l_n$ of an irredundant product of phone calls.}
\end{table}

\section{Open questions} \label{sec:Open}

In view of the extensive computations in
Sections~\ref{sec:Three}--\ref{sec:Five}
and the rather indirect dimension argument in
Section~\ref{sec:Proof}, the most urgent challenge concerning the lossy
gossip monoid is the following.

\begin{que}
Find a purely combinatorial description of a polyhedral fan structure
with support $G_n$. Use this description to prove or disprove the pureness
of dimension $\binom{n}{2}$ and the connectedness in codimension one.
\end{que}

The following question is motivated on the one hand by the fact
that $G_n$ has dimension $\binom{n}{2}$ and on the other hand by
Theorem~\ref{thm:pessimal}, which implies that elements of the ordinary
gossip monoid $G_n(\{0,\infty\})$ have length at most $\binom{n}{2}$.

\begin{que}
Is the length of any element of $G_n$ at most
$\binom{n}{2}$?
\end{que}

Once a satisfactory polyhedral fan for $G_n$ is found, the
somewhat ad-hoc graphs in Sections~\ref{sec:Three} and~\ref{sec:Four}
lead to the following challenge.

\begin{que}
Find a useful notion of optimal realisations of elements of
$G_n$ by graphs with detours, and a notion of {\em tight
spans} of such elements. 
\end{que}

For the relation between tight spans and optimal realisations of metrics
by weighted graphs see \cite[Theorem 5]{Dress84}.

We conclude with two question concerning tropicalisations of orthogonal
groups (Section~\ref{sec:Proof}).

\begin{que}
Is $\Trop(\lieg{O}_n)$ a monoid under tropical matrix multiplication? This
is evident for $n \leq 2$, we have checked it computationally for $n=3$,
and it is open for $n \geq 4$.
\end{que}

\begin{que}
Is it true that $\Trop(\lieg{O}_n) \cap [0,\infty]^{n \times n}$ equals
$\Sym(n) \cdot G_n$? Here the action of $\Sym(n)$ is by permuting
rows. This is true for $n \leq 3$, and open for $n \geq 4$. 
\end{que}

%\bibliographystyle{alpha}
%\bibliography{diffeq}

\end{document}